\documentclass{thielstyle-ams}    
\pdfoutput=1

\usepackage[all]{xypic}

\renewcommand{\brho}{\boldsymbol{\rho}}
\renewcommand{\bnu}{\boldsymbol{\nu}}

\title[Hyperplane arrangements associated to quotient singularities]{Hyperplane arrangements associated to\protect\\symplectic quotient singularities}
\author[G. Bellamy]{Gwyn Bellamy}
\address{School of Mathematics and Statistics, University of Glasgow, University Gardens, Glasgow G12 8QW, Scotland}
\email{gwyn.bellamy@glasgow.ac.uk}
\author[T. Schedler]{Travis Schedler}
\address{Department of Mathematics, South Kensington Campus, Imperial College London, London SW7 2AZ, United Kingdom}
\email{trasched@gmail.com}
\author[U. Thiel]{Ulrich Thiel}
\address{Universität Stuttgart, Fachbereich Mathematik, Institut für Algebra und Zahlentheorie, Lehrstuhl für Algebra, Pfaffenwaldring 57, 70569 Stuttgart, Germany}
\email{mailto:thiel@mathematik.uni-stuttgart.de}

\begin{document}  
\blfootnote{Date: 13.06.2017}
\begin{abstract}
We study the hyperplane arrangements associated, via the minimal model programme, to symplectic quotient singularities. We show that this hyperplane arrangement equals the arrangement of CM-hyperplanes coming from the representation theory of restricted rational Cherednik algebras. We explain some of the interesting consequences of this identification for the representation theory of restricted rational Cherednik algebras. We also show that the Calogero-Moser space is smooth if and only if the Calogero-Moser families are trivial. We describe the arrangements of CM-hyperplanes associated to several exceptional complex reflection groups, some of which are free.    
\end{abstract}
 
\maketitle
\tableofcontents

\section*{Introduction}
\setcounter{section}{1}

The goal of this article, which can be viewed as a sequel to \cite{BellamyNamikawa}, is to study basic properties of the hyperplane arrangements associated, via the minimal model programme, to symplectic quotient singularities. In order to be able to compute these hyperplane arrangements, it seems essential to try and understand how they are related to the representation theory of the associated symplectic reflection algebra. Our results are focused on doing just that. 

We recall Namikawa's theory \cite{Namikawa,Namikawa2,Namikawa3}, applied to our situation. Let $V$ be a finite-dimensional symplectic complex vector space and $\Gamma \subset \mathrm{Sp}(V)$ a finite group. Then the quotient $X := V/ \Gamma$ is a conic symplectic variety. By the minimal model programme, we can choose a projective $\Q$-factorial terminalization $\rho : Y \rightarrow X$.  We fix $\mf{c} := H^2(Y,\C)$. By \cite{Namikawa3} the space $Y$ admits a universal graded Poisson deformation $\bnu : \mc{Y} \rightarrow \mf{c}$ over $\mf{c}$. The quotient $X$ also admits a flat Poisson deformation $\eta:\mf{X} \to \mf{c}$ over $\mf{c}$ such that $\rho$ is the specialization at zero of a morphism $\brho : \mc{Y} \rightarrow \mf{X}$ and  the diagram 
\begin{equation}\label{eq:commtri}
\begin{tikzpicture}
\draw [->] (0.3,0) -- (3.7,0);
\draw [->] (0.2,-0.2) -- (1.8,-1.3);
\draw [->] (3.8,-0.3) -- (2.2,-1.3);
\node at (0,0) {$\mc{Y}$};
\node at (2,0.2) {$\brho$};
\node at (4,0) {$\mf{X} $};
\node at (2,-1.5) {$\mf{c}$};
\node at (0.8,-1) {$\bnu$};
\node at (3.2,-1) {$\eta$};
\end{tikzpicture}
\end{equation}
is commutative.

Assume that there exists a $\Gamma$-stable Lagrangian $\h \subset V$, so that $V = \h \oplus \h^*$ as a symplectic $\Gamma$-module. Then there is a natural action of a one dimensional torus $T$ on $X$. This is the action induced from the action 
\begin{equation}\label{eq:torusaction}
t \cdot (y,x) = (ty,t^{-1} x), \quad \forall \ x \in \h^*, \ y \in \h. 
\end{equation}
In particular, $t \cdot y = t y$ for $y \in \h$. The action is Hamiltonian, and extends to a Hamiltonian action on $\mf{X}$. The fact that $\brho: \mc{Y} \rightarrow \mf{X}$ is Poisson implies that the action lifts to $\mc{Y}$, such that $Y \subset \mc{Y}$ is $T$-stable. Defining a trivial action of $T$ on $\mf{c}$, the maps in the commutative diagram (\ref{eq:commtri}) are equivariant. 

\begin{prop}\label{conj:mainresult}
$|\mc{Y}_{\bk}^T|$ is independent of $\bk \in \mf{c}$. 
\end{prop}

The motivation for proving this proposition comes from the representation theory of restricted rational Cherednik algebras. In this case we are dealing with a finite complex reflection group $\Gamma \subset \mathrm{GL}(\h)$ and we consider $V := \h \oplus \h^*$ with the induced $\Gamma$-action, making $(\Gamma,V)$ a symplectic reflection group.  Note that $|\mf{X}^T_{\bk}|$ is finite for all $\bk$. Let $N := |\mf{X}^T_{\bk'}|$ for some generic $\bk'$. Then $ |\mf{X}^T_{\bk}| \le N$ for all $\bk \in \mf{c}$. Bonnaf\'e-Rouquier \cite[\S 9.7]{BonnafeRouquier} have shown that the set $\mc{E} \subset \mf{c}$ of all $\bk$ such that $ |\mf{X}^T_{\bk}| < N$ is a purely codimension one closed subvariety. In all known examples, $\mc{E}$ is a union of hyperplanes, and these hyperplanes are closely related to the \textit{essential hyperplanes} defining Rouquier's families for cyclotomic Hecke algebras; see \cite{Chloubook} for Rouquier families and \cite{GordonMartinoCM,Mo,BellThiel1} for results and conjectures about this relation. Representation-theoretically, the closed points of $\mf{X}^T_{\bk}$ are in bijection with the blocks of the restricted rational Cherednik algebra $\overline{\H}_{\bk}(\Gamma)$; this is explained further below. Thus, $\mc{E}$ has a purely representation theoretic interpretation. On the other hand, Namikawa has shown that there is a natural hyperplane arrangement $\mc{D} \subset \mf{c}$, having various equivalent geometric definitions \cite{Namikawa3}. The easiest of which is to say that $\mc{D}$ consists of all points $\bk$ such that $\brho_{\bk} : \mc{Y}_{\bk} \rightarrow \mf{X}_{\bk}$ is \textit{not} an isomorphism. We prove that proposition \ref{conj:mainresult} implies:

\begin{thm}\label{thm:mainthm}
$\mc{D} = \mc{E}$. 
\end{thm}

Theorem \ref{thm:mainthm} has several consequences:
\begin{enumerate}
\item $\mc{E}$ is indeed a union of hyperplanes, answering in the affirmative Question 9.8.4 raised by Bonnaf\'e and Rouquier in \cite{BonnafeRouquier}. 
\item $\mc{E}$ is stable under the action of Namikawa's Weyl group $W$, see \S\ref{nam_weyl} and Lemma \ref{nam_weyl_sym}. (This is not true for the essential hyperplanes defining Rouquier's families.)
\item $\mc{E}$ contains the Coxeter arrangement corresponding to $W$.
\item We identify a certain ``fine-structure'' in $\mc{E}$, namely it consists of two types of hyperplanes: those where $\mf{X}_\bk$ has non-terminal singularities (we call them \textit{$T$-hyperplanes}) and those where $\mf{X}_\bk$ is not $\Q$-factorial (we call them \textit{$F$-hyperplanes}).
\item The arrangement $\mc{E}$ is rational i.e. there is a canonical $\Q$-subspace $\mf{c}_{\Q} \subset \mf{c}$ such that $\mc{E} \subset \mf{c}_{\Q}$ and $\mf{c} = \mf{c}_{\Q} \otimes_{\Q} \C$. 
\end{enumerate}

Let $\Irr \Gamma$ denote the set of irreducible complex representations of $\Gamma$. As noted above, the blocks of the restricted rational Cherednik algebra $\overline{\H}_{\bk}(\Gamma)$ define a partition $\Omega_{\bk}(\Gamma)$ of $\Irr \Gamma$ into so-called \textit{Calogero-Moser families}. The families in $\Omega_{\bk}(\Gamma)$ are naturally in bijection with the closed points of $\mf{X}_{\bk}^T$. We say that the Calogero-Moser families are \textit{trivial} if each member of the partition contains only one element. This representation theoretic characterization of $\mf{X}_{\bk}^T$ allows one to use the representation theory of $\overline{\H}_{\bk}(\Gamma)$ to deduce facts about the geometry of $\mf{X}_{\bk}$. Given a parabolic subgroup $\Gamma'$ of $\Gamma$, we let $\bk'$ denote the restriction of $\bk$ to $\Gamma'$. We show that:

\begin{thm}\label{thm:CMfamilytriv}
If the Calogero-Moser families $\Omega_{\bk'}(\Gamma')$ are non-trivial for some parabolic subgroup $\Gamma' \subset \Gamma$, then the Calogero-Moser families $\Omega_{\bk}(\Gamma)$ are non-trivial. 
\end{thm}

This theorem has the following immediate geometric corollary. 

\begin{cor}\label{cor:CM1} 
The Calogero-Moser space $\mf{X}_{\bk}$ is smooth if and only if the Calogero-Moser families $\Omega_{\bk}(\Gamma)$ are trivial. 
\end{cor}

The proof of the above results are direct, and do not use the geometry of the terminalizations. \\

The CM-hyperplanes $\mc{E}$ have been computed in many examples. In the case of $\Gamma = \Z_{\ell} \wr \s_n$, this was done by Martino \cite{Mo} and Gordon \cite{GordonQuiver}, and we recall their description of $\mc{E}$. They have also been computed for many of the irreducible exceptional complex reflection groups by the third author \cite{Thiel:2015jl}, and by Bonnafé and the third author \cite{BT-CM-computational}. We describe the Poincar\'e polynomial of these arrangements and say when they are free. 

\begin{example}
As a concrete example, consider the CM-hyperplane arrangement $\mc{E}$ associated to the exceptional complex reflection group $G_8$. This consists of the following $25$ hyperplanes
$$
-\kappa_{0} + \kappa_{3} = 0, \; -\kappa_{0} + \kappa_{2} = 0, \; \kappa_{0} + \kappa_{2} - 2 \kappa_{3} = 0, \; \kappa_{2} - \kappa_{3} = 0, \quad \kappa_{0} - 3 \kappa_{1} + \kappa_{2} + \kappa_{3} = 0,
$$
$$
-2 \kappa_{0} + \kappa_{2} + \kappa_{3} = 0, \quad  - \kappa_{0} + 2\kappa_{2} - \kappa_{3} = 0, \quad  -\kappa_{0} + \kappa_{1} = 0, \quad  \kappa_{0} + \kappa_{1} - 2 \kappa_{3} = 0, 
$$
$$
\kappa_{1} - \kappa_{3} = 0, \quad  -2 \kappa_{0} + \kappa_{1} + \kappa_{3} = 0, \quad  \kappa_{0} + \kappa_{1} - 3 \kappa_{2} + \kappa_{3} = 0, \quad  \kappa_{0} + \kappa_{1} - 2 \kappa_{2} = 0,
$$
$$
\kappa_{1} - 2 \kappa_{2} + \kappa_{3} = 0, \quad  \kappa_{1} - \kappa_{2} = 0, \quad  \kappa_{0} + \kappa_{1} - \kappa_{2} - \kappa_{3} = 0, \quad  -\kappa_{0} + \kappa_{1} - \kappa_{2} + \kappa_{3} = 0, 
$$
$$
-2 \kappa_0 + \kappa_{1} + \kappa_{2} = 0, \;  \kappa_{0} + \kappa_{1} + \kappa_{2} - 3 \kappa_{3} = 0, \quad  \kappa_{1} + \kappa_{2} -  2\kappa_{3} = 0, \;  -\kappa_{0} + \kappa_{1} + \kappa_{2} - \kappa_{3} = 0, 
$$
$$
-3 \kappa_{0} + \kappa_{1} + \kappa_{2} + \kappa_{3} = 0, \;  - \kappa_{0} + 2 \kappa_{1} - \kappa_{3} = 0, \;  -\kappa_{0} + 2 \kappa_{1} - \kappa_{2} = 0, \; 2\kappa_{1} - \kappa_{2} - \kappa_{3} = 0.
$$ 
This is an arrangement in $\Q^4$ that is stable under the permutation action of the symmetric group $\s_4$. It contains the Coxeter arrangement of type $\mathsf{A}_3$ as a subarrangement, has Poincar\'e polynomial $(13t + 1) (11 t + 1) (t + 1)$, and is free. 
\end{example}

Though our focus is on hyperplane arrangements associated to symplectic quotient singularities, we would like to emphasise that the general theory exists for any singular conic symplectic variety. Thus, via symplectic geometry, one can produce a very large class of hyperplane arrangements. We believe that this rich source of arrangements should be of considerable interest to those interested in the combinatorics of hyperplane arrangements. 

\subsection*{Acknowledgments}

We are very grateful to Ivan Losev for fruitful conversations. We also thank Cédric Bonnaf\'e for explaining to us his joint work with Daniel Juteau, Emily Norton for pointing out a typo in the abstract, and Gunter Malle for several comments. The first author was supported by EPSRC grant EP-N005058-1. The second author was partly supported by NSF grant DMS-1406553.

\section{Symplectic singularities}\label{sec:sympsing}

A variety will mean an integral, separated scheme of finite type over $\C$. By a \textit{resolution of singularities} $\pi : Y \rightarrow X$, we mean a proper birational map $\pi$ from a smooth variety $Y$. If we require $\pi$ to be projective, we will explicitly say so. All vector spaces considered will be complex and finite-dimensional.  

Recall that a \textit{symplectic variety} is a normal variety $X$ over $\C$ such that:
\begin{enumerate}
\item[(a)] the smooth locus of $X$ is equipped with an algebraic symplectic $2$-form $\omega$;
\item[(b)] if $\rho : Y \rightarrow X$ is a resolution of singularities, then $\rho^* \omega$ extends to a regular $2$-form on $Y$.
\end{enumerate}
The resolution $\rho$ is said to be \textit{symplectic} if the extension $\omega'$ of $\rho^* \omega$ to $Y$ is everywhere non-degenerate. In particular, $Y$ is an algebraic symplectic manifold. Since $X$ is normal, the form $\omega$ makes $X$ into a Poisson variety.

A normal variety $Y$ is said to be \textit{$\Q$-factorial} if some multiple of each Weil divisor is Cartier. Since $Y$ is normal there is an embedding $\mathrm{Pic} (Y) \hookrightarrow \mathrm{Cl}(Y)$ and $Y$ is $\Q$-factorial if and only if the quotient group is torsion. In the case where both $\mathrm{Pic} (Y)$ and $\mathrm{Cl}(Y)$ have finite rank, $Y$ is $\Q$-factorial if and only if $\rk \mathrm{Pic} (Y) = \rk \mathrm{Cl}(Y)$. A normal variety $Y$ whose canonical divisor $K_Y$ is $\Q$-Cartier is said to have \textit{terminal singularities} if for any resolution of singularities $\pi : Z \rightarrow Y$, 
$$
K_Z = \pi^*(K_Y) + \sum_E a_E E
$$
with $a_E > 0$, where the sum is over all exceptional divisors $E$ of $\pi$. Here $\pi^*(K_Y) := \frac{1}{n} \pi^*(n K_Y)$ for some $n \gg 0$ such that $n K_Y$ is Cartier.

We say that the affine symplectic singularity $X$ has a \textit{conic} $\Cs$-action of weight $\ell$ if the $\Cs$-action makes $\C[X]$ an $\N$-graded connected algebra and the form $\omega$ is homogeneous of weight $\ell >0$. In particular, $X$ has a unique fixed point under a conic action; this is the cone point $o \in X$. By the minimal model programme \cite{BCHM}, we can fix a $\Q$-factorial terminalization $\rho : Y \rightarrow X$. Since $\rho$ is crepant, it is a Poisson morphism between symplectic varieties. Then $\mf{c} := H^2(Y;\C)$ is given the structure of an affine space. As noted previously, the universal graded Poisson deformation $\mc{Y}$ of $Y$ has base $\mf{c}$. As shown in \cite{Namikawa2}, there is a finite group $W$, called \textit{Namikawa's Weyl group} that acts faithfully as a (real) reflection group on $\mf{c}$. Namikawa \cite{Namikawa,Namikawa2} has shown that there is a universal graded Poisson deformation $\beta:\mc{X} \rightarrow \mf{c} / W$ of $X$ (so that $\beta^{-1}(0) \simeq X$). The conic action on $X$ lifts to a $\Cs$-action on $Y$ making $\rho$ equivariant. Let $\bnu : \mc{Y} \rightarrow \mf{c}$ be the universal graded Poisson deformation of $Y$. Then the action of $\Cs$ on $Y$ lifts to $\mc{Y}$ and $\bnu$ is equivariant for the weight $\ell$ action of $\Cs$ on $\mf{c}$. Summarizing, as in (2) of \cite{Namikawa3}, one has a commutative diagram 
$$
\xymatrix{
\mc{X} \ar[d]_{\beta} & \mc{Y} \ar[d]^{\bnu} \ar[l] \\
\mf{c} / W  & \ar[l] \mf{c}
}
$$
with all maps being $\Cs$-equivariant. Set $\mf{X} = \mc{X} \times_{\mf{c} / W} \mf{c}$ and write $\boldsymbol{\rho} : \mc{Y} \rightarrow \mf{X}$ for the corresponding morphism over $\mf{c}$. For each closed point $\bk \in \mf{c}$, we write $\bs{\rho}_{\bk} :  \mc{Y}_{\bk} \rightarrow \mf{X}_{\bk}$ for the corresponding morphism. Let $\mc{D}$ denote the set of points in $\mf{c}$ where $\bs{\rho}_{\bk}$ is \textit{not} an isomorphism. By \cite{Namikawa3}  the set $\mc{D}$ is precisely the set of points $\bk$ such that either a) $\mf{X}_{\bk}$ does not have terminal singularities; or b) is not $\Q$-factorial. 

Assume that there exists a connected one-dimensional torus $T$ acting by Hamiltonian automorphisms on $X$, commuting with the action of $\Cs$. This action lifts to $\mc{X}$, acting trivially on $\mf{c} / W$. The following proposition by Namikawa \cite{NamikawaPoissondeformations} will be important later.

\begin{prop}\label{prop:locallytrivial}
For all $y \in Y$, $d_y \bnu : T_y \mc{Y} \rightarrow T_{0} \mf{c}$ is surjective.  
\end{prop}

\begin{proof}
First, we note by Corollary A.10 of \cite{NamikawaPoissondeformations} that $Y^{\mathrm{an}}$ is $\Q$-factorial. Restricting $\mc{Y}$ to the formal neighbourhood $\widehat{\mf{c}}_0$ of $0$ in $\mf{c}$, Theorem 17 of \cite{NamikawaPoissondeformations} says that $\mc{Y}$ is locally trivial in the analytic topology. In particular, if $\widehat{\mc{Y}}_y$, resp. $\widehat{Y}_y$, denotes the formal neighbourhood of $y$ in $\mc{Y}$, resp. in $Y$, this means that for each $y \in Y$, there is an isomorphism of formal schemes 
\begin{equation}\label{eq:smoothfactor}
\xymatrix{
\widehat{\mc{Y}}_y \ar[rr]^{\sim} \ar[dr]_{\bnu} & & \widehat{Y}_y \ \times \ \widehat{\mf{c}}_0 \ar[dl]^{\mathrm{pr}_2} \\
 & \widehat{\mf{c}}_0
}
\end{equation}
The proposition follows. 
\end{proof} 

For a variety $Z$, we write $\{ Z_i \}_{i = 1,2,\ds}$ for the \textit{singular stratification}. This is defined inductively by setting $Z_0 = Z_{\sm}$, the smooth locus of $Z$, and $Z_i$ to be the smooth locus of $(Z \smallsetminus \bigcup_{j < i} Z_j)_{\mathrm{red}}$. Here $(-)_{\mathrm{red}}$ is taking the reduced scheme structure. Then each $Z_i$ is a smooth variety, though in general disconnected with components of different dimension. Since it is canonically defined, each $Z_i$ is $G$-stable for any group $G$ acting on $Z$. Let $\bnu^{(i)}$ denote the restriction of $\bnu$ to the singular stratum $\mc{Y}_i$ of $\mc{Y}$. 

\begin{prop}\label{prop:smoothstrata} 
The morphism $\nu^{(i)} : \mc{Y}_i \rightarrow \mf{c}$ is smooth. 
\end{prop}

\begin{proof}
The singular stratification is compatible with passing to formal neighbourhoods, i.e. $(\widehat{\mc{Y}}_y)_i = \widehat{(\mc{Y}_i)_y} =: \widehat{\mc{Y}}_{i,y}$. We note that if $C$ is any connected component of $\mc{Y}_i$, then $C \cap Y \neq \emptyset$. This follows from the fact that $C$ is $\Cs$-stable and the limit as $t \rightarrow 0$ of any point in $C$ under this action exists, belonging to $C \cap Y$. Let $y \in C \cap Y$. Then, the fact that each stratum $\mc{Y}_i$ is canonically defined implies that the isomorphism of formal schemes (\ref{eq:smoothfactor}) restricts to a commutative diagram 
$$
\xymatrix{
\widehat{\mc{Y}}_{i,y} \ar[rr]^{\sim} \ar[dr]_{\bnu} & & \widehat{Y}_{i,y} \times \widehat{\mf{c}}_0 \ar[dl]^{\mathrm{pr}_2} \\
 & \widehat{\mf{c}}_{0}.
}
$$
In particular, $d_y \bnu^{(i)}$ is surjective. Since $C$ is smooth, standard results, e.g. \cite[III, Proposition 10.4]{Hartshorne} imply that $\bnu^{(i)}$ is smooth at $y$.  This implies that $\Omega_{\mc{Y}_i / \mf{c}}$ is locally free in some open neighbourhood of $Y_i$. Thus, the locus where $\Omega_{\mc{Y}_i / \mf{c}}$ is not locally free is some proper closed subset of $\mc{Y}_i$. Since $\bnu^{(i)}$ is an equivariant map, this locus is also $\Cs$-stable. Every point in $\mc{Y}_i$ has a limit in $Y_i$. Thus, we conclude that $\Omega_{\mc{Y}_i / \mf{c}}$ is locally free.  
\end{proof}  

\begin{cor}\label{cor:smoothmorphism}
If $Y$ is smooth then the morphism $\bnu : \mc{Y} \rightarrow \mf{c}$ is smooth. 
\end{cor} 

Regarding the morphisms $\beta : \mc{X} \rightarrow \mf{c} / W$ and $\eta : \mf{X} \rightarrow \mf{c}$, it can happen in examples that $\mc{X}$ is smooth (though this is certainly not always the case), in which case $\beta$ is \textit{not} a smooth morphism. On the other hand, $d_x \eta$ will always be surjective for $x \in X \subset \mf{X}$, and hence $\mf{X}$ is never smooth. This behavior is already apparent for $X$ the Kleinian singularity $\C^2 / \Z_2$.

\subsection{Hyperplane arrangements}\label{sec:hyperplanes}

Recall that $\mc{D}$ is the locus of points $\bk \in \mf{c}$ such that $\mf{X}_{\bk}$ does not have $\Q$-factorial terminal singularities, i.e. either $\mf{X}_{\bk}$ is not $\Q$-factorial, or it does not have terminal singularities. By \cite{Namikawa3} the set $\mc{D}$ is a union of finitely many hyperplanes, which we call the \textit{Namikawa hyperplanes} of $X$. 

\begin{defn}
A Namikawa hyperplane $L \subset \mc{D}$ is said to be a \textit{$T$-hyperplane} if $\mf{X}_{\bk}$ has non-terminal singularities for all $\bk \in L$. Otherwise, $L$ is said to be an $F$-hyperplane. 
\end{defn}

\begin{remark}
We note that if a generic point $\bk$ of a Namikawa hyperplane $L$ is such that $\mf{X}_{\bk}$ has non-terminal singularities, then $\mf{X}_{\bk}$ has non-terminal singularities for all $\bk \in L$. One can see this from the fact that since $\mf{X}_{\bk}$ is a symplectic variety, it has terminal singularities if and only if the singular locus has codimension at least $4$, see \cite{NamikawaNote}. If $\bk$ is a generic point of an $F$-hyperplane, then $\mf{X}_{\bk}$ is not $\Q$-factorial. However, there will in general be other points $\bk' \in L$ for which $\mf{X}_{\bk'}$ is $\Q$-factorial. That is, the locus in $\mf{c}$ where $\mf{X}_{\bk}$ is $\Q$-factorial is neither open nor closed (but it is dense). This can easily be seen for the quotient singularities we consider below. 
\end{remark} 

Recall that Namikawa's Weyl group $W$ acts on $\mf{c}$, and it is shown in \cite{Namikawa2} that the subset $\mc{D}$ is $W$-stable. Therefore $W$ permutes the Namikawa hyperplanes of $X$. 

\subsection{$T$-actions}
Assume now that $T$ is a torus acting by Hamiltonian automorphisms on $\mc{Y}$, $\mf{X}$ etc. such that $\mf{X}^T_{\bk}$ is a finite set for all $\bk \in \mf{c}$. Since $\mc{Y}_{\bk} = \mf{X}_{\bk}$ for generic $\bk$, this implies that $\mc{Y}^T_{\bk}$ is a finite set for generic $\bk$. The following proposition is the key to proving the main result of the article. 

\begin{prop}\label{prop:Ysmoothfinite}
$|\mc{Y}^T_{\bk}| < \infty$ is independent of $\bk$. 
\end{prop}

\begin{proof}
The group $T$ preserves each of the strata $\mc{Y}_i$. Since the scheme $\mc{Y}_i$ is smooth, $T_y \mc{Y}^T_i = (T_y \mc{Y})^T$ for all $y \in \mc{Y}_i^T$. By proposition \ref{prop:smoothstrata}, the morphism $\bnu^{(i)}$ is $T$-equivariant and smooth. Recalling that $T$ acts trivially on $\mf{c}$, this implies that $d \bnu^{(i)} |_{\mc{Y}^T_i}$ is also surjective. We deduce, as in proposition \ref{prop:smoothstrata}, that $\bnu^{(i)} |_{\mc{Y}^T_i}$ is smooth. The generic fiber of $ \bnu^{(i)} |_{\mc{Y}^T_i}$ is finite. Thus, every fiber of $ \bnu^{(i)} |_{\mc{Y}^T_i}$ is finite, and $ \bnu^{(i)} |_{\mc{Y}^T_i}$ is \'etale. This implies that 
$$
\left| \left( \bnu^{(i)} |_{\mc{Y}^T_i} \right)^{-1} \right| = | \mc{Y}_{\bk}^T \cap \mc{Y}_i |
$$
is independent of $\bk$. Since $| \mc{Y}_{\bk}^T| = \sum_i | \mc{Y}_{\bk}^T \cap \mc{Y}_i |$, the result follows. 
\end{proof}

\begin{remark}
In the case where $\mc{Y}$ is not smooth, it is still true that $d_y \bnu |_{\mc{Y}^T} : T_y \mc{Y}^T \rightarrow T_0 \mf{c}$ is surjective for $y \in Y$, but only when $\mc{Y}^T$ is considered as a non-reduced scheme. It is not clear that  $d_y \bnu |_{\mc{Y}^T_{\mathrm{red}}}$ is surjective. 
\end{remark}

\subsection{Symplectic leaves}

Being symplectic varieties, the spaces $Y$ and $X$, as above, have a finite stratification by symplectic leaves. These leaves can be characterized as the connected components of the rank stratification. That is, we say that $x \in X_p$ if and only if the rank of the Poisson bracket at $x$ is $p \in \{ 0, \ds, \dim X \}$. Then the symplectic leaves of $X$ are the connected components of smooth, locally closed subvarieties $X_p$. By Lemma 3.1 (4) and Proposition 3.7 of \cite{PoissonOrders}, the closure of a leaf is a union of leaves. The following results will be needed later. 

\begin{lem}\label{lem:dimleaves}
Let $X$ and $Y$ be symplectic varieties, with $X$ affine and assume that $f : Y \rightarrow X$ is a dominant Poisson morphism. If $p \in \mc{L} \subset Y$ and $f(p) \in \mc{M} \subset X$, with $\mc{L}$ and $\mc{M}$ symplectic leaves, then 
$$
\dim \mc{L} \ge \dim \mc{M}.
$$ 
\end{lem}

\begin{proof}
Let $\mf{m} \subset \mc{O}_{Y,p}$ and $\mf{n} \subset \mc{O}_{X,f(p)}$ be the corresponding maximal ideals. Since $f$ is dominant, we have an injective morphism $f^* : \mc{O}_{X,\pi(p)} \rightarrow \mc{O}_{Y,p}$ of Poisson algebras. The bracket defines a skew-symmetric form on $T^*_p Y = \mf{m} / \mf{m}^2$ by 
$$
\{ \overline{u}, \overline{v} \} :=  \{ u, v \} \ \mathrm{mod} \ \mf{m}, 
$$
and similarly for $T^*_{f(p)} X$. The natural map $f^* : \mf{n} / \mf{n}^2 \rightarrow \mf{m} / \mf{m}^2$ intertwines forms. Thus, it induces a morphism 
$$
T^*_{f(p)} \mc{M} = (\mf{n} / \mf{n}^2) / \Ker \{ - , - \} \rightarrow T^*_p \mc{L} = (\mf{m} / \mf{m}^2) / \Ker \{ - , - \}
$$
of symplectic vector spaces. This must be injective. Since the leaf $\mc{M}$ is smooth, $\dim \mathcal{M} = \dim T_{f(p)}^*\mathcal{M}$ and similarly for $\mathcal{L}$. The result follows. 
\end{proof}

\begin{thm}\label{thm:unionleaves}
Let $X$ and $Y$ be symplectic varieties, with $X$ affine and assume that $f : Y \rightarrow X$ is a projective birational Poisson morphism. The locus in $X$ where $f$ is not an isomorphism is a union of symplectic leaves. 
\end{thm}

\begin{proof}
Recall that the smooth locus $X_{\sm}$ and $Y_{\sm}$ are symplectic leaves. We begin by showing that $f$ is an isomorphism over $X_{\sm}$. Since $X$ is normal, it suffices to show that $|f^{-1}(x)| = 1$ for $x \in X_{\sm}$. Also note that the conditions of Zariski's main theorem are satisfied. Thus, if $|f^{-1}(x)| > 1$ then $\dim |f^{-1}(x)| \ge 1$ since the fibers of $f$ are connected. Choose a leaf $\mc{L} \subset Y$ such that $\mc{L} \cap f^{-1}(x)$ is open (and non-empty) in $f^{-1}(x)$. Then Lemma \ref{lem:dimleaves} implies that $\dim \mc{L} \ge \dim X = \dim Y$. Thus, $\mc{L} = Y_{\sm}$. Moreover, the proof of Lemma \ref{lem:dimleaves} shows that $d_y f : T_y Y_{\sm} \rightarrow T_x X_{\sm}$ is surjective at $y \in  \mc{L} \cap f^{-1}(x)$. Hence $f$ is étale at $y$. But $f$ is also birational. Thus, $f$ is an isomorphism there. 

If the statement of the theorem is not true, then there must exist some leaf $\mc{M} \subset X$ such that $f$ is an isomorphism over a generic point of $\mc{M}$, but there exists $s_0 \in \mc{M}$ for which $\dim f^{-1}(m_0) > 0$. Replacing $X$ by a sufficiently small affine open neighborhood of $m_0$ and $Y$ by the preimage of this neighborhood, we may assume that $\mc{M}$ is closed in $X$. Moreover, we may assume that $\mc{M}$ is the unique closed leaf and all other leaves have strictly greater dimension. Now if $m \in \mc{M}$ is generic, then there exists a leaf $p \in \mc{L} \subset Y$ with $\dim \mc{L} = \dim \mc{M}$ and $f(p) = m$. Notice that $\mc{L}$ is closed in $Y$. Otherwise, there exists $\mc{L}' \subset \overline{\mc{L}} \smallsetminus \mc{L}$ with $\dim \mc{L}' < \dim \mc{L}$. But applying Lemma \ref{lem:dimleaves} to any point in $\mc{L}'$ gives a contradiction on the minimality of $\mc{M}$. Thus, $f(\mc{L})$ is a closed (since $f$ is proper) irreducible subvariety of $X$ contained in $\mc{M}$. Since $\mc{M}$ is also closed and irreducible and $\dim f(\mc{L}) = \dim \mc{M}$, we conclude that $f(\mc{L}) = \mc{M}$. Since $\mc{L}$ and $\mc{M}$ are smooth, Lemma  \ref{lem:dimleaves} implies that $f |_{\mc{L}} : \mc{L} \rightarrow \mc{M}$ is a smooth morphism. Since it is generically an isomorphism, we conclude that it is everywhere an isomorphism. 

Now let $m_0 \in \mc{M}$ as before. Take $\mc{L}_1 \subset Y$ a leaf such that $\mc{L}_1 \cap f^{-1}(m_0)$ is open (and non-empty) in $f^{-1}(m_0)$. We wish to show that $\mc{L}_1 = \mc{L}$. This would contradict $\dim f^{-1}(m_0) > 0$, and hence show that $f^{-1}(\mc{M}) = \mc{L}$. Since we have assumed that $f$ is an isomorphism over a generic point of $\mc{M}$, it suffices to show that $\mc{M} \cap f(\mc{L}_1)$ is dense in $\mc{M}$. Assume otherwise. Then $Z := \overline{\mc{M} \cap f(\mc{L}_1)}$ is a proper closed subvariety of $\mc{M}$. The morphism $f |_{\mc{L}_1 \cap f^{-1}(\mc{M})}$ factors as $\mc{L}_1 \cap f^{-1}(\mc{M}) \rightarrow Z \hookrightarrow \mc{M}$. Choose $p_1 \in \mc{L}_1 \cap f^{-1}(\mc{M})$. Then the induced map $T^*_{f(p_1)} \mc{M} \rightarrow T^*_{p_1} \mc{L}_1$ would have a non-trivial kernel. However, as explained in the proof of Lemma \ref{lem:dimleaves}, this cannot happen since $f$ is Poisson.
\end{proof}

\section{Symplectic quotient singularities}

In this section we turn to the particular class of conic symplectic varieties that we are interested in, namely that of symplectic quotient singularities. We let $(V,\omega)$ be a finite dimensional symplectic vector space and $\Gamma \subset \mathrm{Sp}(V)$ a finite group. An element $s \in \Gamma$ is said to be a \textit{symplectic reflection} if $\mathrm{rk} (1- s) = 2$. We say that $\Gamma$ is a symplectic reflection group if it is generated by its set $\mc{S}$ of reflections. Though it is not strictly necessary, we will assume throughout that $\Gamma$ is a symplectic reflection group. We denote by $\Irr \Gamma$ the set of (isomorphism classes of) complex irreducible $\Gamma$-modules.

\subsection{Namikawa Weyl group} \label{nam_weyl}

Recall that Namikawa's Weyl group $W$ acts on $\mf{c}$. Here we describe how to make the group, and the action, precise in the case $X = V / \Gamma$. 

 First, abusing terminology, we will say that a subspace $H \subset V$ is a \textit{symplectic hyperplane} if $\dim H = \dim V - 2$ and the restriction of $\omega$ to $H$ is non-degenerate. For each $s \in \mc{S}$, the subspace $\Ker (1 - s)$ is a symplectic hyperplane. Let $\mc{A}$ denote all symplectic hyperplanes that arise in this way. Then $\mc{A}$ is a finite set, and $\Gamma$ acts on $\mc{A}$ in the natural way. For each $H \in \mc{A}$, the subgroup $\Gamma_H$ of $\Gamma$ that acts pointwise trivially on $H$ is a minimal parabolic of $\Gamma$ in the sense of \cite[\S 1.1]{BellamyNamikawa}. The non-identity elements of $\Gamma_H$ are precisely the reflections of $\Gamma$ with hyperplane $H$. If $\mc{B}$ is the set of conjugacy classes of minimal parabolics of $\Gamma$, as in \textit{loc. cit.}, then there is a natural bijection $\mc{A} / \Gamma \rightarrow \mc{B}$, $[H] \mapsto [\Gamma_H]$. As explained in \cite[\S1.1]{BellamyNamikawa} for any $H \in \mathcal{A}$, the minimal parabolic subgroup $\Gamma_H$ is isomorphic to a subgroup of $\mathrm{SL}(2,\C)$. Hence, via the McKay correspondence, there is an associated Weyl group $(W_H,\mf{c}_{H})$ of simply laced type. Here $\mf{c}_H$ is the reflection representation of $W_H$. The pair $(W_H,\mf{c}_H)$ is independent of the choice of embedding $\Gamma_H \hookrightarrow \mathrm{SL}(2,\C)$. 

There is a natural linear action $\zeta : \Xi_H \rightarrow \mathrm{GL}(\mf{c}_H)$ of the quotient $\Xi_H := N_{\Gamma}(\Gamma_H) / \Gamma_H$ on $\mf{c}_{H}$ by Dynkin diagram automorphisms; see \cite[\S2.1]{BellamyNamikawa}. Let 
$$
W_H^{\Xi_H} := \{ w\in W_H \ | \ \zeta(x) w \zeta(x)^{-1} = w, \ \forall \ x \in \Xi_H \}
$$
denote the centralizer of of $\Xi_H$ in $W_H$.  Now, Namikawa's Weyl group is 
\begin{equation}\label{eq:Namweyl}
W = \prod_{[H] \in \mc{A}/\Gamma} W_H^{\Xi_H} \;.
\end{equation}
By \cite[Theorem 1.3]{BellamyNamikawa}, there is an isomorphism of $W$-modules
$$
\mf{c} \simeq \bigoplus_{[H]\in \mathcal{A}/\Gamma} \mf{c}_{H}^{\Xi_H}.
$$
In particular, $W_H^{\Xi_H}$ acts as a reflection group on $\mf{c}_{H}^{\Xi_H}$. Let $R_H \subset \mf{c}_H^*$ denote the root system of $W_H$. Since the action of $\Xi_H$ on $W_H$ is by Dynkin automorphisms, one can identify the Weyl group $(W_H^{\Xi_H},\mf{c}_{H}^{\Xi_H})$ with the Weyl group of the folded Dynkin diagram. Since we will only consider cases where $\Xi_H = 1$, we do not elaborate on this further; see \cite{Slodowy} for details. 

Assume that $\zeta(\Xi_H) = \{ 1 \}$. Since we have identified $\mf{c} = \bigoplus_{H \in \mc{A} / \Gamma} \mf{c}_H$ as a direct sum of reflection representations, for each root $\alpha \in R_{H}$, we get a root hyperplane $L_{\alpha} \subset \mf{c}$.

\begin{lem}\label{lem:terminalsym2}
The root hyperplanes $L_{\alpha}$ are precisely the $T$-hyperplanes in $\mc{D}$. 
\end{lem}

\begin{proof}
Recall that the $T$-hyperplanes are those hyperplanes where $\mf{X}_{\bk}$ does not have terminal singularities. In other words, $\mf{X}_{\bk}$ has at least one symplectic leaf of codimension two. Then the result follows from Losev's \cite[Theorem 1.3.2]{LosevSRAComplete} (see \cite[Appendix A]{smoothsra} for a different formulation), noting the fact we have assumed that $\zeta(\Xi_H) = \{ 1 \}$, which implies that the restriction map $\mf{c} \rightarrow \mf{c}_H$ is surjective. 
\end{proof}

\subsection{Symplectic reflection algebras}

The group $\Gamma$ acts on $\mc{S}$ by conjugation. Fix a $\Gamma$-invariant function $\bc : \mc{S} \rightarrow \C$. For each $s \in \mc{S}$ let $\omega_s$ be the skew-symmetric form on $V$, whose restriction to $\Ker (1 - s)$ is zero and equals $\omega$ on $\mathrm{Im} ( 1- s)$. Then the \textit{symplectic reflection algebra} $\H_{\bc}(\Gamma)$ at $t = 0$  associated to $\Gamma$ at $\bc$ is the quotient of $T V \rtimes \Gamma$, where $TV$ is the tensor algebra of $V$, by the relations 
\begin{equation}\label{eq:SRAfirst}
[v,w] = \sum_{s \in \mc{S}} \bc(s) \omega_s(v,w) s, \quad \forall \ v,w \in V. 
\end{equation}
In applications, it is convenient to give a different presentation of these relations. There is a permutation action of $\Gamma$ on the space of tuples $\bk = (\bk_{H,\eta})_{H \in \mc{A}, \eta \in \Irr \Gamma_H}$ of complex numbers via $\,^g \bk = (\bk_{g(H),\,^g\eta})_{H,\eta}$ for $g \in \Gamma$, where $\,^g \eta$ is the $\Gamma_{g(H)}$-module with the same underlying space as $\eta$ but with action $h \cdot v := g^{-1}hgv$ for $h \in \Gamma_{g(H)}$ and $v \in \eta$. We only consider tuples such that 
\begin{enumerate}
\item[(a)] $\bk$ is $\Gamma$-invariant, i.e. $\bk_{H,\eta} = \bk_{g(H),{}^g \eta}$ for all $g \in \Gamma$, $H \in \mc{A}$ and $\eta \in \Irr \Gamma_H$, 
\item[(b)] $\sum_{\eta \in \Irr \Gamma_H} \bk_{H,\eta} \delta_{H,\eta} = 0$ for all $H \in \mc{A}$, where $\delta_{H,\eta} := \dim \eta$. 
\end{enumerate}
Such tuples are indexed by pairs $(\lbrack H \rbrack, \eta)$ with $\lbrack H \rbrack \in \mathcal{A}/\Gamma$ and $\eta \in \Irr \Gamma_H$ for a fixed representative $H$ of $\lbrack H \rbrack$. For now, write $K$ for the (affine) space of all such $\bk$-tuples. 

Recall that $(W_H,\mf{c}_H)$ is the Weyl group associated to $\Gamma_H$ via the McKay correspondence. The space $\mf{c}_H$ has a basis given by a set $\Delta^{\vee}_H$ of simple coroots of $W_H$. Via the McKay correspondence, this set $\Delta^{\vee}_H$ is in bijection with $\Irr_* \Gamma_H := \Irr \Gamma_H \smallsetminus \{ \mathrm{triv} \}$, the set of non-trivial irreducible representations of $\Gamma_H$. Write $\eta \mapsto \alpha_{\eta}$ for the bijection $\Irr_* \Gamma_H \rightarrow \Delta^{\vee}_H$. Thus, we may define a map from the space $K$ to $\bigoplus_{[H] \in \mc{A} / \Gamma} \mf{c}_H$ by 
$$
\bk \mapsto \sum_{[H] \in \mc{A} / \Gamma} \sum_{\eta \in \Irr_* \Gamma_H} \bk_{\eta} \alpha_{\eta}. 
$$
The following was explained in \cite[\S 3.2]{BellamyNamikawa}. 

\begin{lem}
The above map defines an isomorphism $K \stackrel{\sim}{\longrightarrow} \mf{c} = \bigoplus_{[H] \in \mc{A} / \Gamma} \mf{c}_H^{\Xi_H}$ of affine spaces. 
\end{lem}

From now on, we identify $K = \mf{c}$. Notice that the restriction that $\bk$ lie in the closed subspace $\mf{c}_{H}^{\Xi_H} \subset \mf{c}_{H}$ is equivalent to the restriction (a) in the definition of $\bk$. Now, for $H \in \mc{A}$, let $\omega_H$ denote the skew-symmetric form on $V$ whose restriction to $H$ is trivial and equals $\omega$ on $H^{\perp}$. Then, for each $\bk \in \mf{c}$, the \textit{symplectic reflection algebra} $\H_{\bk}(\Gamma)$ at $t = 0$ is the quotient of $T V \rtimes \Gamma$ by the relations 
\begin{equation}\label{eq:SRAsecond}
[v,w] = \sum_{H \in \mc{A}} \omega_H(v,w) \sum_{\eta \in \Irr \Gamma_H} \frac{\bk_{H,\eta}}{\delta_{H,\eta}} e_{H,\eta}, \quad \forall \ v,w \in V, 
\end{equation}
where 
$$
e_{H,\eta} := \frac{\delta_{H,\eta}}{|\Gamma_H |} \sum_{g \in \Gamma_H} \chi_{\eta}(g^{-1}) g
$$
is the central idempotent of $\C \Gamma_H$ corresponding to $\eta$ with $\chi_{\eta}$ being the character of $\eta$. The two presentations are related by 
\begin{equation}\label{eq:cvsk}
\bc(s) = \frac{1}{|\Gamma_H|} \sum_{\eta \in \Irr \Gamma_H} \bk_{H,\eta} \chi_{\eta}(s^{-1}),
\end{equation}
for $s \in \Gamma_H \smallsetminus \{ 1 \}$. 

Let $Z(\H_{\bk}(\Gamma))$ denote the centre of $\H_{\bk}(\Gamma)$. One can also consider the generic symplectic reflection algebra $\H_{\mf{c}}(\Gamma)$ defined over $\C[\mf{c}]$. Its centre $ \mathsf{Z}_{\mf{c}}(\Gamma)$ is flat over $\C[\mf{c}]$. If, on the other hand, we choose $Y$ a $\Q$-factorial terminalization of $V / \Gamma$ then, as in section \ref{sec:sympsing}, we get a flat family $\mf{X}$ over $\mf{c}$ deforming $V / \Gamma$. Then \cite[Theorem 1.4]{BellamyNamikawa} says:

\begin{thm}
There is an isomorphism of Poisson varieties 
$$
\begin{tikzpicture}
\draw [->] (0.2,0) -- (3,0);
\draw [->] (0.2,-0.2) -- (1.8,-1.3);
\draw [->] (3.8,-0.3) -- (2.2,-1.3);
\node at (0,0) {$\mf{X}$};
\node at (2,0.2) {$\sim$};
\node at (4,0) {$\Spec   \mathsf{Z}_{\mf{c}}(\Gamma)$};
\node at (2,-1.5) {$\mf{c}$};
\end{tikzpicture}
$$
over $\mf{c}$. 
\end{thm}

When we wish to show that $\mf{X}$ is the universal deformation associated to $\Gamma$, we write $\mf{X}(\Gamma)$. 

\subsection{Example: wreath products}

In this section we consider the key example where $\Gamma$ equals the wreath product $G \wr \s_n := G^n \rtimes \s_n$, for some \textit{non-trivial} finite subgroup $G$ of $\mathrm{SL}(2,\C)$. Here $\s_n$ is the symmetric group on $n$ letters and $\Gamma$ acts on $V = (\C^2)^n$ in the obvious way. If $s_{i,j}$ is the transposition swapping $i$ and $j$ and for $\gamma \in G$, $\gamma_i := (1, \ds, \gamma, \ds, 1) \in G^n$, then define $\gamma_{i,j} := s_{i,j} \gamma_i^{-1} \gamma_j$. The symplectic reflections in $\Gamma$ are $\gamma_{i,j}$ for $\gamma \in G$ and $i \neq j$ and $\gamma_i$ for $\gamma \in G \smallsetminus \{ 1 \}$. The $\gamma_{i,j}$ form a single conjugacy class and $\gamma_i$ is conjugate to $\rho_j$ if and only if $\gamma$ is conjugate to $\rho$ in $G$. This means that the set $\mc{A} / \Gamma = \{ [H_1], [H_2] \}$ has two elements, where $H_1 = \Ker (1 - s_{1,2})$ and $H_2 = \Ker (1 - \gamma_1)$ for any $\gamma \in G \smallsetminus \{ 1 \}$. In particular, $\Gamma_{H_1} = \s_2 = \{ 1, s_{1,2} \}$ and $\Gamma_{H_2} = \{ \gamma_1 \ | \ \gamma \in G \} \simeq G$. Thus, the function $\bc$ can be encoded as $(c_1,\underline{c})$, where $c_1 = \bc(s_{1,2})$ and $\underline{c}$ is a class function on $G$, vanishing on the identity, such that $\bc(\gamma_i) = \underline{c}(\gamma)$. Then the defining relations become
$$
[v,w] = c_1 \sum_{i,j} \omega_{\gamma_{i,j}}(v,w) \gamma_{i,j} + \sum_{i = 1}^n \sum_{\gamma \in G \smallsetminus \{ 1 \} } \underline{c}(\gamma) \omega_{\gamma_i}(v,w) \gamma_i. 
$$ 
Let $\bk_i := \bk_{H_i}$. Since $\bk_{i,\mathrm{triv}} = - \sum_{\eta \in \Irr_* \Gamma_{H_i}} \delta_{H_i,\eta} \bk_{i,\eta}$, we omit it from the notation so that $\bk \in \mf{c}_{H_1} \oplus \mf{c}_{H_2}$. Here $\mf{c}_{H_1} = \{ \bk_{1,\mathrm{sgn}}  \in \C \} \simeq \C$ is the reflection representation for $W_{H_1} = \s_2$, and $\mf{c}_{H_2}$ is the reflection representation for $W_{H_2}$, which is the Weyl group associated to $G$ via the McKay correspondence. By equation (\ref{eq:cvsk}), we deduce that 
$$
2 c_1 = \bk_{1,\mathrm{triv}} - \bk_{1,\mathrm{sgn}}, \quad |G| \underline{c}(\gamma) = \sum_{\eta\in \Irr G} \bk_{2,\eta}  \chi_{\eta}(\gamma^{-1}). 
$$
Recall that $R_{H_1} \subset \mf{c}_{H_1}^*$ and $R_{H_2} \subset \mf{c}_{H_2}^*$ are the corresponding root systems. Let $\langle - , - \rangle : \mf{c}_H^* \times \mf{c}_H \rightarrow \C$ denote the natural pairing. The following theorem is due originally to Martino \cite{Mo} and Gordon \cite{GordonQuiver}. 

\begin{thm}\label{thm:hyperplaneswreath}
The Calogero-Moser space $\mf{X}_{\bk}$ is smooth if and only if 
$$
i |G| \langle \alpha , \bk_{1,\mathrm{sgn}} \rangle + 2 j \langle \beta , \bk_2 \rangle \neq 0, 
$$
for all $\alpha \in R_{H_1}, \beta \in R_{H_2}, \ i \in \{ 0, \pm 1 \}$ and $ j \in \{ -(n-1), \ds, n-1 \}$. 
\end{thm}

\begin{proof}
Let $c= \sum_{\gamma \neq 1} \underline{c}(\gamma)$. Using the fact that $\sum_{\eta \in \Irr \Gamma_H} \bk_{H,\eta} \delta_{H,\eta} = 0$, one can calculate that $\Tr_{\eta} c = \bk_{2,\eta}$. Then the parameter $\lambda$ defined in \S 6.7 of \cite{Mo} is given by  
$$
\lambda_{\eta} = \bk_{\eta}, \quad \lambda_{\mathrm{triv}} = \bk_{\mathrm{triv}} - \frac{|G|}{4} (\bk_{1,\mathrm{triv}} - \bk_{1,\mathrm{sgn}})
$$
which implies that 
$$
\lambda_{\infty} = \frac{ n |G|}{4}  (\bk_{1,\mathrm{triv}} - \bk_{1,\mathrm{sgn}}).
$$
Then the claim follows by repeating the argument given in the proof of \cite[Lemma 4.4]{GordonQuiver}.  
\end{proof}

Theorem \ref{thm:hyperplaneswreath} implies that the hyperplanes $\bk_{1,\mathrm{sgn}} = 0$ and  $\langle \beta , \bk_2 \rangle = 0$, for $\beta \in R_{H_2}$, are the $T$-hyperplanes for $\Gamma$, where $\mf{X}_{\bk}$ has a symplectic leaf of codimension two, and the hyperplanes
$$
i |G| \langle \alpha , \bk_{1,\mathrm{sgn}} \rangle + 2 j \langle \beta , \bk_2 \rangle = 0,
$$
for all $\alpha \in R_{H_1}, \beta \in R_{H_2}, \ i \in \{\pm 1 \}$ and $ j \in \{ -(n-1), \ds, n-1 \} \smallsetminus \{ 0 \}$, are the $F$-hyperplanes, where $\mf{X}_{\bk}$ fails (at least generically) to be $\Q$-factorial.

\section{Rational Cherednik algebras}

Let $\h$ be a finite-dimensional complex vector space and $\Gamma \subset \mathrm{GL}(\h)$ be a finite reflection group. Then $V = \h \oplus \h^*$ is a symplectic vector space and $\Gamma$ acts on $V$ as a symplectic reflection group. Each symplectic hyperplane in $\mc{A}$ is of the form $L \oplus L'$ for $L \subset \h$ a reflection hyperplane, and $L' \subset \h^*$. Thus, we can identify $\mc{A}$ with the set of reflecting hyperplanes in $\h$. If $H = L \oplus L'$, then choose $\alpha_H^{\vee} \in L$ and $\alpha_H \in L'$ such that $\langle \alpha_H, \alpha_H^{\vee} \rangle := \alpha_H(\alpha_H^{\vee}) \neq 0$. Define, for each symplectic hyperplane $H$ of $\Gamma$, the form $(\cdot,\cdot)_H: \h \times \h^* \to \C$ by 
$$
(y,x)_H = \frac{\langle x , \alpha_H^{\vee} \rangle \langle \alpha_H,y \rangle}{\langle \alpha_H,\alpha_H^{\vee}  \rangle} \;. 
$$
Then the defining relations for the rational Cherednik algebra are
\[
[y,y'] = 0 \;, \quad [x,x'] = 0 \;,
\]
\[
[y,x] = \sum_{H \in \mc{A}} (y,x)_H \sum_{i = 0}^{|\Gamma_H| - 1} \bk_{H,i} e_{H,i}, \quad \forall y,y' \in \h, \ x,x' \in \h^* \;.
\]
Here  
$$
e_{H,i} := \frac{1}{|\Gamma_H|} \sum_{s \in \Gamma_H} \det(s |_{\mf{h}} )^{-i} s. 
$$
In order to better describe the action of Namikawa's Weyl group, we introduce new variables $\kappa_{H,i}$, where $[H] \in \mc{A} / \Gamma$ and $i = 0, \ds, |\Gamma_H | - 1$. Set
$$
 \kappa_{H,i +1} - \kappa_{H,i} := \frac{1}{|\Gamma_H|} \bk_{H,|\Gamma_H| - i}, \quad \sum_{i = 0}^{|\Gamma_H| -1}  \kappa_{H,i } = 0,
$$
so that the defining relations for the rational Cherednik algebra become
\begin{align*}
[y,x] & = \sum_{H \in \mc{A}} (y,x)_H \sum_{s \in \Gamma_H \smallsetminus \{ 1 \} } \left( \sum_{j = 0}^{|\Gamma_H|-1} \det(s)^j (\kappa_{H,j+1} - \kappa_{H,j}) \right) s\\
 & = \sum_{H \in \mc{A}} |\Gamma_H| (y,x)_H \sum_{j = 0}^{|\Gamma_H|-1} (\kappa_{H,j+1} - \kappa_{H,j}) e_{H,-j}, \quad \forall \ x \in \h^*, y \in \h. 
\end{align*}
If $\s_H := \s_{|\Gamma_H|}$ the symmetric group on $| \Gamma_H |$ letters, then define an action of $\s_H$ on $\mf{c}$ by $\sigma(\kappa_{H',i}) = \kappa_{H',i}$ if $[H] \neq [H']$ in $\mc{A} / \Gamma$ and 
$$
\sigma(\kappa_{H,i}) = \kappa_{H,\sigma(i)}.  
$$

\begin{lem} \label{nam_weyl_sym}
The Namikawa Weyl group of $\Gamma$ is 
$$
\prod_{[H] \in \mc{A}/\Gamma} \s_{H}.
$$
\end{lem}

\begin{proof}
The group $\Gamma$ is a complex reflection group. Therefore the group $\Gamma_H$, for $H \in \mc{A}$, is a cyclic group and the corresponding Weyl group is $\s_H$. Hence, in the notation of section \ref{nam_weyl}, it suffices to show that $\s_H^{\Xi_H} = \s_H$. Equivalently, $\Xi_H$ acts trivially on $\mf{c}_H$. As noted in section A of \cite{BMR}, the normalizer of $\Gamma_H$ is equal to its centralizer $C := C_{\Gamma}(\Gamma_H)$. Thus, $\Xi_H = C/ \Gamma_H$. As explained in section 2.2 of \cite{BellamyNamikawa}, the action of $\Xi_H$ on $\mf{c}_H$ is computed as follows:
\begin{itemize}
\item[a)] $\Xi_H$ acts on $V_H / \Gamma_H$, where $V = V_H \oplus V^{\Gamma_H}$ is the decomposition as $\Gamma_H$-modules;
\item [b)] This action lifts (uniquely) to a minimal resolution of singularities $Y' \rightarrow V_H / \Gamma_H$; 
\item[c)] This induces an action on $H^2(Y';\C) \simeq \mf{c}_H$.
\end{itemize}
Decompose $V_H = \h_H \oplus \h^*_H$ so that $V = (\h_H \oplus \h_H^*) \oplus (\h \oplus \h^*)^{\Gamma_H}$ as $\Gamma_H$-modules. Then $C$ acts on $\h_H$ and $\h_H^*$. Since $\dim \h_H = 1$, the action of $C$ factors through the faithful action of a cyclic group $\Z_k$ on $\h_H$. This is the restriction of the action of a one-dimensional torus $T' \simeq \Cs$ acting by dilations. Then $T'$ acts by Hamiltonian automorphisms on $V_H$, the action descending to $V_H / \Gamma_H$. The action of $T'$ lifts to Hamiltonian automorphisms on $Y'$. Moreover, the action of $\Z_k$ on $Y'$ is given via the inclusion $\Z_k \subset T'$. Since $T'$ is connected, it acts trivially on $H^2(Y';\C)$. Therefore $\Z_k$ acts trivially on $\mf{c}_H$. We deduce that $\Xi_H$ acts trivially on $\mf{c}_H$. 
\end{proof}

\subsection{Calogero-Moser families} In the case where $\Gamma$ is a complex reflection group, the subalgebras $\C[\h]^{\Gamma}$ and $\C[\h^*]^{\Gamma}$ are contained in the centre $\mathsf{Z}_{\bk}(\Gamma)$ of $\H_{\bk}(\Gamma)$. The embedding $\C[\h]^{\Gamma} \otimes \C[\h^*]^{\Gamma} \hookrightarrow \mathsf{Z}_{\bk}(\Gamma)$ defines a finite surjective morphism 
$$
\Upsilon_{\bk} := \pi_1 \times \pi_2 : \mf{X}_{\bk} \rightarrow \h / \Gamma \times \h^* / \Gamma.
$$
Recall that  a subgroup $\Gamma'$ of $\Gamma$ is a parabolic subgroup if and only if there exists some $x \in \h$ such that $\Gamma'$ is the stabilizer of $x$ with respect to $\Gamma$.  If $\Gamma'$ is a parabolic subgroup then $(\Gamma')$ denotes its conjugacy class. Let $\h^{\Gamma'}_{\reg}$ denote the locally closed subset of $\h$ consisting of all points with stabilizer $\Gamma'$. Its closure is $\h^{\Gamma'}$. The image $\h^{(\Gamma')} / \Gamma$ of $\h^{\Gamma'}$ in $\h / \Gamma$ only depends on $(\Gamma')$. The rank of $(\Gamma')$ is defined to be $\dim \h - \dim \h^{\Gamma'}$. 

\begin{lem}\label{lem:easystab}
Let $\Gamma'$ be a subgroup of $\Gamma$. Then $\Gamma'$ is a stabilizer subgroup of $(\Gamma,\h)$ if and only if it is a stabilizer subgroup of $(\Gamma,\h \oplus \h^*)$. 
\end{lem}

\begin{proof}
Let $\Gamma'$ be a stabilizer subgroup of $(\Gamma,\h)$. Then $\Gamma'$ is clearly a stabilizer subgroup of $(\Gamma,\h \oplus \h^*)$. Conversely, assume that $\Gamma'$ is a stabilizer subgroup of $\h \oplus \h^*$. Then there exists some $x \in \h \oplus \h^*$ such that $\Gamma' = \Stab_{\Gamma}(x)$. The vector $x$ is a generic point of $(\h \oplus \h^*)^{\Gamma'}$. Note that $(\h \oplus \h^*)^{\Gamma'} = \h^{\Gamma'} \oplus (\h^*)^{\Gamma'}$. Let $y \in \h^{\Gamma'}$ be generic. If there exists $g \in \Gamma \smallsetminus \Gamma'$ such that $g(y) = y$ then $\Ker (1 - g)_{\h} \supsetneq \h^{\Gamma'}$. This implies that $\Ker (1 - g)_{\h^*} \supsetneq (\h^*)^{\Gamma'}$. Thus, $g(x) = x$ and hence $g \in \Gamma'$; a contradiction.  
\end{proof}

The following key result will imply Theorem \ref{thm:CMfamilytriv} and Corollary \ref{cor:CM1} of the introduction. 

\begin{thm}\label{thm:Upsilonleaf}
Let $\mc{L}$ be a symplectic leaf of $\mf{X}_{\bk}$. Then there exists a conjugacy class of parabolic subgroups $(\Gamma')$ of $\Gamma$ such that $\Upsilon_{\bk} \left( \overline{\mc{L}} \right) = \h^{(\Gamma')}/ \Gamma \times (\h^*)^{(\Gamma')} / \Gamma$. 
\end{thm}

The proof of Theorem \ref{thm:Upsilonleaf} is given in section \ref{sec:thmUproof}.
  
\begin{cor}\label{cor:fixedleaf}
Let $\mc{L}$ be a symplectic leaf in $\mf{X}_{\bk}$. Then $\overline{\mc{L}}^{T} \neq \emptyset$. 
\end{cor}

\begin{proof}
Notice that $0$ is the only $T$-fixed point in $\h / \Gamma \times \h^* / \Gamma$. The map $\Upsilon_{\bk}$ is equivariant such that $\Upsilon_{\bk}^{-1}(0) = \mf{X}_{\bk}^T$. Therefore $\overline{\mc{L}}^T \neq \emptyset$ if and only if $0 \in \Upsilon_{\bk} \left( \overline{\mc{L}} \right)$. By Theorem \ref{thm:Upsilonleaf}, $\Upsilon_{\bk} \left( \overline{\mc{L}} \right)$ equals $\h^{(\Gamma')}/ \Gamma \times (\h^*)^{(\Gamma')} / \Gamma$ for some parabolic subgroup $\Gamma'$, and this contains zero. 
\end{proof}

We recall from \cite{Baby} that the blocks of the restricted rational Cherednik algebra $\overline{\H}_{\bk}(\Gamma)$ define a partition of $\Irr \Gamma$ into \textit{Calogero-Moser families}. The Calogero-Moser families are naturally in bijection with the closed points of $\mf{X}_{\bk}^T$. We say that the Calogero-Moser families are \textit{trivial} if each member of the partition contains only one element. The following is a strengthening (in characteristic zero) of \cite[Theorem 1.3]{MoPositiveChar}.

\begin{cor}
The Calogero-Moser space $\mf{X}_{\bk}$ is smooth if and only if the Calogero-Moser families of $\Gamma$ are trivial. 
\end{cor}

\begin{proof}
The fact that the Calogero-Moser families are trivial when $\mf{X}_{\bk}$ is smooth is a standard result. Conversely, if $\mf{X}_{\bk}$ is not smooth then there exists a leaf $\mc{L}$ in the singular locus of $\mf{X}_{\bk}$. Corollary \ref{cor:fixedleaf} says that there is a fixed point in $\overline{\mc{L}}$. In particular, there is a $T$-fixed point in the singular locus of $\mf{X}_{\bk}$. The block of the Calogero-Moser family supported at this fixed point is non-trivial. 
\end{proof}

\subsection{Proof of Theorem \ref{thm:Upsilonleaf}}\label{sec:thmUproof}

Suitably reformulated, Proposition 4.8 of \cite{Cuspidal} says that: 

\begin{lem}\label{lem:leafinter}
Let $\mc{L}$ be a symplectic leaf in $\mf{X}_{\bk}$ of dimension $2 \ell$. There exists a unique conjugacy class $(\Gamma')$ of parabolic subgroups of $\Gamma$ with rank $(\Gamma') = n - \ell$ such that $\overline{\pi_1(\mathcal{L})} = \h^{(\Gamma')}/\Gamma$.
\end{lem}

There is an exhaustive filtration $\{ \mc{F}_i \H_{\bk}(\Gamma) \ | \ i \in \Z_{\ge 0} \}$ on the rational Cherednik algebra given by placing $\h,\h^* \subset \H_{\bk}(\Gamma)$ in degree one and $\Gamma$ in degree zero. Set $\mc{F}_i \ZH_{\bk}(\Gamma) = \ZH_{\bk}(\Gamma) \cap \H_{\bk}(\Gamma)$. If $\mf{p}$ is the Poisson prime ideal of $\ZH_{\bk}(\Gamma)$ defining $\overline{\mc{L}}$ then $\gr_{\mc{F}}(\mf{p})$ is a Poisson prime in $\C[\h \times \h^*]^\Gamma$ by \cite[Theorem 2.8]{MartinoAssociated}.  This defines a map 
$$
\Omega : \mathrm{Symp } \ \mf{X}_{\bk} \rightarrow \mathrm{Symp} \ (\h \times \h^*)/\Gamma,
$$ 
where $\mathrm{Symp } \ \mf{X}_{\bk}$, resp. $\mathrm{Symp} \ (\h \times \h^*)/\Gamma$, denotes the set of symplectic leaves in $\mf{X}_{\bk}$, resp. in $(\h \times \h^*)/\Gamma$. 

\begin{lem}\label{lem:leftrightW}
Let $\mc{L}$ be a symplectic leaf in $\mf{X}_{\bk}$ of dimension $2 \ell$. If $\overline{\pi_1(\mathcal{L})} = \h^{(\Gamma')}/\Gamma$ then $\overline{\pi_2(\mathcal{L})} = (\h^*)^{(\Gamma')}/\Gamma$. 
\end{lem}

\begin{proof}
Let $\mf{p}$ be the Poisson prime ideal defining $\overline{\mc{L}}$. Repeating the argument given in the proof of \cite[Proposition 4.2]{BrownChangtong}, we have  
\begin{equation}\label{eq:primeintersect}
\gr_{\mc{F}}(\mf{p}) \cap \C[\h]^\Gamma= \mf{p} \cap \C[\h]^\Gamma, \quad \gr_{\mc{F}}(\mf{p}) \cap \C[\h^*]^\Gamma= \mf{p} \cap \C[\h^*]^\Gamma. 
\end{equation} 
Since $\overline{\pi_1(\mc{L})} = V(\mf{p} \cap \C[\h]^\Gamma)$, equations (\ref{eq:primeintersect}) imply that it suffices to show that if $\mc{L}$ is a leaf of $(\h \times \h^*)/ \Gamma$ then $\overline{\pi_1(\mc{L})} = \h^{(\Gamma')}/\Gamma$ implies that $\overline{\pi_2(\mc{L})} = (\h^*)^{(\Gamma')}/\Gamma$. Since $\Upsilon_0(\overline{\mc{L}})$ is closed and irreducible of codimension $2 \ell$, it suffices to show that $\h^{(\Gamma')}/ \Gamma \times (\h^*)^{(\Gamma')} / \Gamma$ is contained in $\Upsilon_{0}(\overline{\mc{L}})$. 

As shown in \cite[Proposition 7.4]{PoissonOrders}, there exists a parabolic subgroup $\Gamma''$ of $(\Gamma,\h \times \h^*)$ such that $I(\overline{\mc{L}}) = J(\Gamma'')$, where $J(\Gamma'') = \C[\h \times \h^*]^\Gamma \cap I(\Gamma'')$ and $I(\Gamma'')$ is the ideal generated by $\{ x - g(x) \ | \ x \in \h \times \h^*, \ g \in \Gamma'' \}$. By Lemma \ref{lem:easystab}, $\Gamma''$ is also a parabolic subgroup of $(\Gamma,\h)$. Since $\{ x - g(x) \ | \ x \in \h^*, \ g \in \Gamma'' \}$ is contained in $I(\Gamma'')$, the irreducible variety $\overline{\pi_1(\mc{L})}$ is contained in $\h^{(\Gamma'')}/\Gamma$. However, these two spaces have the same dimension. Thus, 
$$
\h^{(\Gamma')}/\Gamma = \overline{\pi_1(\mc{L})} = \h^{(\Gamma'')}/\Gamma, 
$$
and hence $(\Gamma') = (\Gamma'')$. Since $\overline{\pi_2(\mc{L})} = (\h^*)^{(\Gamma'')}/\Gamma$ too, the result follows.  
\end{proof}

Theorem \ref{thm:Upsilonleaf} is a direct consequence of Lemmata \ref{lem:leafinter} and \ref{lem:leftrightW}. Notice that we have shown 
$$
\Upsilon_{\bk} \left( \overline{\mc{L}} \right) = \h^{(\Gamma')}/ \Gamma \times (\h^*)^{(\Gamma')} / \Gamma = \Upsilon_{0}(\overline{\Omega(\mc{L})}).
$$

\section{The equality $\mc{E} = \mc{D}$}\label{sec:Tchered}

Recall from (\ref{eq:torusaction}) that we defined a Hamiltonian $T$-action on $(\h \times \h^*) / \Gamma$. The relations above make it clear that the rational Cherednik algebra $\H_{\bk}(\Gamma)$ is graded, with $\deg (x) = -1$, $\deg (y) = 1$ and $\deg(g) = 0$ for $x \in \h^*, y \in \h$ and $g \in \Gamma$. By restriction, $\mathsf{Z}_{\bk}(\Gamma)$ is $\Z$-graded and hence $T$ acts on $\mf{X}_{\bk}$. This action is Hamiltonian and when $\bk = 0$, so that $\mf{X}_{\bk} = (\h \times \h^*) / \Gamma$, it agrees with the action defined by (\ref{eq:torusaction}). The locus $\mc{E}$ was defined as follows: let $N := |\mf{X}_{\bk'}^T|$ for some generic $\bk'$. Then $N \le | \Irr \Gamma |$ and $\bk \in \mc{E}$ if and only if $|\mf{X}_{\bk}^T | < N$. Fix a $\Q$-factorial terminalization $Y$ of $V / \Gamma$. 

\begin{thm}\label{thm:fiberlessthan}
$\mc{E} = \mc{D}$. 
\end{thm}

\begin{proof}
By \cite[Corollary 1.6]{BellamyNamikawa}, $\mf{X}$ is the affinization of $\mc{Y}$ and the map $\brho : \mc{Y} \rightarrow \mf{X}$ is $T$-equivariant. For generic $\bk$ the morphism $\brho_{\bk}$ is an ismorphism. This implies that $|\mf{X}_{\bk}^T | = |\mc{Y}^T_{\bk} |$ generically. Since $|\mf{X}_{\bk}^T | = N$ for generic $\bk$, we deduce that $|\mc{Y}^T_{\bk} | = N$ for generically. Then Proposition \ref{prop:Ysmoothfinite} says that $|\mc{Y}^T_{\bk} | = N$ for all $\bk$. Thus, it suffices to prove that the hyperplane arrangement $\mc{D}$ equals the set 
$$
\{ \bk \in \mf{c} \ | \  |\mf{X}_{\bk}^T | < |\mc{Y}^T_{\bk} | \}. 
$$
Clearly, if $\brho_{\bk}$ is an isomorphism then $|\mf{X}_{\bk}^T | = |\mc{Y}^T_{\bk} |$. Therefore, we just need to show that if $\brho_{\bk}$ is not an isomorphism then $|\mf{X}_{\bk}^T | < |\mc{Y}^T_{\bk} |$. Let $\mc{Z} := \brho^{-1}(\mf{X}^T)$. Restriction defines a projective map $\brho_{\mc{Z}} : \mc{Z} \rightarrow \mf{X}^T$ over $\mf{c}$. This is still $T \times \Cs$-equivariant with $T$ acting trivially on the base. Hence each fiber of $\brho_{\mc{Z},\bk} : \mc{Z}_{\bk} \rightarrow \mf{X}^T_{\bk}$ is a projective $T$-stable variety with only finitely many $T$-fixed points. The Bia\l ynicki-Birula decomposition implies that $\brho_{\mc{Z},\bk}^{-1}(x)$ admits a paving by affine spaces, with the number of affine pieces equal to $\dim_{\C} H^{\idot}\left(\brho_{\mc{Z},\bk}^{-1}(x); \C \right)$. The fact that $\brho_{*} \mc{O}_{\mc{Y}} = \mc{O}_{\mf{X}}$ and $\rho$ is projective implies, by Zariski's main theorem, that the fibers of $\brho$ are connected. Therefore, $H^{0}\left(\brho_{\mc{Z},\bk}^{-1}(x); \C \right) = \C$ and $H^{\idot}\left(\brho_{\mc{Z},\bk}^{-1}(x); \C \right) = H^{0}\left(\brho_{\mc{Z},\bk}^{-1}(x); \C \right)$ if and only if $|\brho_{\mc{Z},\bk}^{-1}(x)| = 1$. This implies that 
$$
|\mf{X}_{\bk}^T | = \sum_{x \in \mf{X}_{\bk}^T} \dim H^{0}\left(\brho_{\mc{Z},\bk}^{-1}(x); \C \right) \le \sum_{x \in \mf{X}_{\bk}^T} \dim H^{\idot}\left(\brho_{\mc{Z},\bk}^{-1}(x); \C \right)
$$
with equality if and only if $|\brho_{\mc{Z},\bk}^{-1}(x)| = 1$ for all $x \in \mf{X}_{\bk}^T$. Thus, we must show that $\brho_{\bk}$ is not an isomorphism if and only if $\brho_{\mc{Z},\bk}$ is not an isomorphism. Equivalently, if $\mc{W} \subset \mf{X}_{\bk}$ is the locus over which $\brho_{\bk}$ is not an isomorphism, then $\mc{W}^T \neq \emptyset$. By Theorem \ref{thm:unionleaves}, $\mc{W}$ is a union of leaves. Therefore Corollary \ref{cor:fixedleaf} says that $\mc{W}^T \neq \emptyset$, as required. 
\end{proof}

It follows from Theorem \ref{thm:fiberlessthan} that $\mc{E}$ is a union of hyperplanes. We will refer to these hyperplanes as the \textit{Calogero-Moser hyperplanes}, or CM-hyperplanes for short. In the examples below, we will abuse notation and let $\mc{E}$ denote both the closed subset of $\mf{c}$ and the set of all hyperplanes that are contained in $\mc{E}$. Now we note four important consequences of Theorem \ref{thm:fiberlessthan}. 

\begin{cor}\label{prop:conjcons}\hfill
\begin{enum_thm}
\item The root hyperplanes $L_{\alpha}$, for $\alpha \in R_{H}$ and $H \in \mc{A} / \Gamma$, i.e. the Coxter arrangement of $W$, are contained in $\mc{E}$. 
\item The CM-hyperplanes are permuted by Namikawa's Weyl group. 
\item The $F$-hyperplanes are precisely the root hyperplanes $L_{\alpha}$, for $\alpha \in R_{H}$ and $H \in \mc{A} / \Gamma$.
\item The arrangement $\mc{E}$ is contained in the rational subspace $H^2(Y;\Q)$ of $\mf{c}= H^2(Y;\C)$. 
\end{enum_thm}
\end{cor}

\begin{proof}
We have already shown in Lemma \ref{lem:terminalsym2} that the  root hyperplanes $L_{\alpha}$ are precisely the $T$-hyperplanes in $\mc{D}$. This implies (a) and (c). Similarly, (b) is a general fact about Namikawa's arrangement; see section \ref{sec:hyperplanes}. Finally, under the identifcation $\mc{E} = \mc{D}$, (d) is precisely \cite[Main Theorem]{Namikawa3}.
\end{proof}

\begin{remark}
Theorem \ref{thm:fiberlessthan}, together with Lemma \ref{lem:terminalsym2} above, implies that if $\mf{X}_{\bk}$ has a symplectic leaf of codimension two, then $|\Omega_{\bk}(\Gamma)| < N$ is not maximal. We do not know how to show this directly. 
\end{remark} 

\begin{remark}
It is expected that there is a close relationship between the essential hyperplanes of Rouquier families, coming from cyclotomic Hecke algebras, and  the CM-hyperplanes; see \cite{Thiel:2013gx,GordonMartinoCM,Mo} and the references therein. However, the two arrangements are not equal, and the CM-arrangement seems to be ``better  behaved'' in general. For instance,  the CM-hyperplanes are permuted by Namikawa's Weyl group. This is very rarely the case for the arrangement of essential hyperplanes. As a concrete example, if we take $\Gamma$ to be the exceptional complex reflection group $G_{10}$ then there are only $81$ essential hyperplanes defining Rouquier's families. This arrangement is not stable under $\s_3 \times \s_4$. However, there are $111$ CM-hyperplanes in $\mc{E}$ and this arrangement is stable under $\s_3 \times \s_4$. 
\end{remark}

\section{Open questions}

This work raises a number of natural questions and problems, which we believe to be worthy of attack. 

\begin{enumerate}
\item Compute the hyperplane arrangements $\mc{D}$ for all symplectic reflection groups. 
\item For which symplectic reflection groups is the arrangement $\mc{D}$ free? Inductively free? Is the complement $K(\pi,1)$? 
\item Given a symplectic reflection group $\Gamma$ and $\bk \in \mf{c}$, compute $\Pic (\mf{X}_{\bk}(\Gamma))$ and $\Cl(\mf{X}_{\bk}(\Gamma))$. 
\item For a given $\Gamma$, describe $H^{\idot}(\mf{c} \smallsetminus \mc{D}, \C)$ as a graded $W$-module. 
\item Give a presentation of $\pi_1(\mf{c} \smallsetminus \mc{D})$.
\end{enumerate}

\begin{remark}
Problem (3) should be a fun exercise even for partial deformations of Kleinian singularities. Problem (5) is important in describing the derived equivalences that are expected between different $\Q$-factorial terminalizations of $(\h \times \h^*) / \Gamma$. 
\end{remark}

\section{Examples}

In this section we describe the CM-hyperplanes $\mc{E}$ for a large class of examples. The importance of the Orlik-Solomon algebra $H^{\idot}(\mf{c} \smallsetminus \mc{D};\C)$ was explained in \cite{BellamyNamikawa}. In particular, the integer $E := \frac{1}{|W|} \dim H^{\idot}(\mf{c} \smallsetminus \mc{D};\C)$ computes the number of $\Q$-factorial terminalizations admitted by $V / \Gamma$. We compute $E$ in these example. More generally, we compute the Poincar\'e polynomial 
$$
P_{\Gamma}(t) = \sum_{i \ge 0} t^i \dim H^{i}(\mf{c} \smallsetminus \mc{D};\C). 
$$

\begin{lem}\label{lem:Poicare}
The degree of $P_{\Gamma}(t)$ equals $\dim \mf{c}$.
\end{lem}

\begin{proof}
By Corollary \ref{prop:conjcons}, the Coxeter hyperplanes $L_{\alpha}$, for $\alpha \in R_H$ and $H \in \mc{A}$, all belong to $\mc{D}$. Therefore,  
$$
\bigcap_{L \in \mc{D}} L \subset \bigcap_{\alpha \in R} L_{\alpha} = 0. 
$$
Then the result follows from Theorem 2.47 and Definition 2.48 of \cite{OrlikTeraoBook}. 
\end{proof}

Lemma \ref{lem:Poicare} implies that the simplest arrangements come from those groups where $\dim \mf{c} = 1$. A complex reflection group is said to be a \textit{$2$-reflection group} if every reflection has order two and there is a single conjugacy class of reflections. 

\begin{prop}\label{prop:2reflgroup}
If $\Gamma$ is a $2$-reflection group then $\mf{c} \smallsetminus \mc{D} = \C^{\times}$, $P_{\Gamma}(t) = 1 + t$ and $(\h \times \h^*) / \Gamma$ admits a unique $\Q$-factorial terminalization. 
\end{prop}

\begin{proof}
Since there is only one conjugacy class of reflections and they all have order two, we immediately get $W = \s_2$, $\mf{c} = \C$ and $\mf{c} \smallsetminus \mc{D} = \C^{\times}$. Then it is clear that $P_{\Gamma}(t) = 1 + t$.  
\end{proof}

Noting that a $2$-reflection group is necessarily irreducible, the following follows easily from the Shephard-Todd classification. 

\begin{lem}
The $2$-reflection groups are $G(m,m,n)$ for $n > 2$, $G(p,p,2)$ with $p$ odd, and the exceptional groups 
\begin{align*}
G_{12}, G_{22}, H_3 = G_{23}, G_{24}, G_{27}, G_{29}, H_4 = G_{30}, \\ G_{31}, G_{33}, G_{34}, E_6 = G_{35}, E_7 = G_{36}, E_8 = G_{37}.
\end{align*}
\end{lem}

\begin{proof}
The exceptional groups can be checked by computer, or from the relevant tables. This leaves the infinite series $G(m,p,n)$. By \cite{Read}, the number of conjugacy classes of reflections is $m/p$ if $n>2$ or if $n=2$ and $p$ is odd, otherwise ($n=2$ and $p$ even) it is $m/p + 1$. Therefore, we have a single conjugacy class of reflections if and only if $m=p$ and $n>2$ or if $m=p$, $n=2$ and $p$ is odd. Moreover, if $m/p=1$, then all reflections have order $2$.
\end{proof}

\begin{example}
Let $\Gamma = \Z_{\ell}$, acting on $\h = \C$. Then the CM-hyperplanes are all of type $T$, and hence precisely the Coxeter arrangement  of type $\mathsf{A}$, i.e. 
$$
\bk_{1,i} + \cdots + \bk_{1,j} = 0, \quad \forall 1 \le i \le j \le \ell - 1.
$$
Equivalently, this is $\kappa_{1,i} - \kappa_{1,j} = 0$ for $0 \le j < i \le \ell-1$. 
\end{example}

\begin{example}
Let $\Gamma = \Z_{\ell} \wr \s_n$ for $n > 1$. Then 
\begin{align*}
\mf{c} & = \mf{c}_{H_1} \times \mf{c}_{H_2} \\
& = \left\{ (\bk_{1,0},\bk_{1,1};\bk_{2,0},\bk_{2,1}, \ds, \bk_{2,\ell-1}) \in \C^{\ell+2} \ | \  \bk_{1,0} + \bk_{1,1} = \sum_i \bk_{2,i} = 0 \right\}, \\
 &  =  \left\{ (\kappa_{1,0},\kappa_{1,1};\kappa_{2,0},\kappa_{2,1}, \ds, \kappa_{2,\ell-1}) \in \C^{\ell+2} \ | \  \kappa_{1,0} + \kappa_{1,1} = \sum_i \kappa_{2,i} = 0 \right\},
\end{align*}
with $\frac{1}{2} \bk_{1,1} = \kappa_{1,0} - \kappa_{1,1}$ and $\frac{1}{\ell} \bk_{2,\ell-i} = \kappa_{2,i+1} - \kappa_{2,i}$.  The $T$-hyperplanes are 
$$
\bk_{1,1} = 0, \quad \bk_{2,i} + \cdots + \bk_{2,j}= 0, \quad \forall \ i < j, 
$$
and the $F$-hyperplanes are 
$$
k \ell \bk_{1,1} + 2 m (\bk_{2,i} + \cdots + \bk_{2,j}) = 0,
$$
where $1 \le i \le j \le \ell-1$, $k \in \{ \pm 1 \}$ and $m \in \{ \pm 1, \ds, \pm (n-1) \}$. Equivalently, they are 
$$
\kappa_{1,0} - \kappa_{1,1} = 0, \quad \kappa_{2,i} - \kappa_{2,j}= 0, \quad \forall \ 0 \le i < j \le \ell-1, 
$$
and  $k (\kappa_{1,0} - \kappa_{1,1})  + m (\kappa_{2,i}  -  \kappa_{2,j}) = 0$. 
\end{example}

\subsection{Dihedral groups}

We consider, as an example, the dihedral groups $D_m = G(m,m,2)$. If $m$ is odd then $D_m$ is a $2$-reflection group, and is covered by Proposition \ref{prop:2reflgroup}. Therefore we assume that $m$ is even. In this case there are two conjugacy classes of reflections. Hence $\mc{A} / D_m = \{ [H_1], [H_2] \}$. Thus, $\mf{c} = \{ (\bk_{1,1},\bk_{2,1}) \in \C^2 \}$ with $W = \s_2 \times \s_2$ acting in the obvious way. It is shown in \cite[Section 6.10]{BellamyThesis} that the CM-hyperplanes are
$$
\mc{E} = \{ \bk_{1,1} = 0, \bk_{2,1} = 0, \bk_{1,1} + \bk_{2,1} = 0,  \bk_{1,1} - \bk_{2,1} = 0 \}. 
$$
By Lemma \ref{lem:terminalsym2}, the $T$-hyperplanes for $D_m$ are $ \bk_{1,1} = 0$ and $\bk_{2,1} = 0$. Therefore Theorem \ref{thm:fiberlessthan} says that $\mf{X}_{\bk}$ is \textit{not} $\Q$-factorial for a generic point of the hyperplanes $\bk_{1,1} + \bk_{2,1} = 0$ or $\bk_{1,1} - \bk_{2,1} = 0$. Moreover, for each $\bk \in \mf{c} \smallsetminus \mc{E}$, the Calogero-Moser space $\mf{X}_{\bk}$ is $\Q$-factorial and terminal. One can see directly that it has terminal singularities since there are no codimension two leaves. In fact, looking at the description of the leaves given in \cite[Table 1]{Cuspidal}, we see that $\mf{X}_{\bk}$ has a single isolated singularity $p$. This point is $T$-fixed. Then the equality $\mc{E} = \mc{D}$ implies that the local ring $\mathsf{Z}_{\bk}(D_m)_p$ has torsion class group for all $\bk \in  \mf{c} \smallsetminus \mc{E}$.\footnotemark.\footnotetext{When $m = 3$, $D_m$ is the exceptional Weyl group of type $G_2$. We have been told by Cedric Bonnaf\'e that he and Daniel Juteau have shown that the singularity $(\mf{X}_{\bk},p )$ is equivalent to $(\overline{\mc{O}}_{\mathrm{min}},0)$, where $\mc{O}_{\mathrm{min}}$ is the minimal nilpotent orbit in $\mf{sp}(4)$. Here $\bk \in \mf{c} \smallsetminus \mc{E}$ and $p$ the unique singular point. Thus, it also follows from \cite[Proposition 2.10]{FuNilpotentResolutions} the local ring $\mathsf{Z}_{\bk}(D_3)_p$ has torsion class group. They have also shown that if $\Gamma = \s_2 \wr \Z_2$ and $\bk$ a generic point of the $F$-hyperplane $\kappa_{1,0} - \kappa_{1,1} = 0$, then $(\mf{X}_{\bk},p)$ is equivalent to $(\overline{\mc{O}}_{\mathrm{min}},0)$, where $\mc{O}_{\mathrm{min}}$ is the minimal nilpotent orbit in $\mf{sl}(3)$. Here $p$ is the unique singular point of $\mf{X}_{\bk}$. Since $T^* \mathbb{P}^2$ resolves this singulartity, \cite[Corollary 1.3]{FuNilpotentResolutions} also implies that the local ring $\mathsf{Z}_{\bk}(\s_2 \wr \Z_2)_p$ does not have torsion class group.} The Orlik-Solomon algebra $H^{\idot}(\mf{c} \smallsetminus \mc{E}, \C)$ is the quotient of the exterior algebra $\wedge^{\idot}(x_1,x_2,x_3,x_4)$ by the ideal generated by 
\begin{align*}
\partial(x_1,x_2,x_3) & = x_2 \wedge x_3 - x_1 \wedge x_3 + x_1 \wedge x_2,  \\
\partial(x_1,x_3,x_4) & = x_3 \wedge x_4 - x_1 \wedge x_4 + x_1 \wedge x_3,  \\
\partial(x_1,x_2,x_4) & = x_2 \wedge x_4 - x_1 \wedge x_4 + x_1 \wedge x_2,  \\
\partial(x_2,x_3,x_4) & = x_3 \wedge x_4 - x_2 \wedge x_4 + x_2 \wedge x_3.
\end{align*}
This implies that 
$$
x_3 \wedge x_4 = x_2 \wedge x_4 = x_1 \wedge x_4 - x_1 \wedge x_2,
$$
and $x_1 \wedge x_3 = x_1 \wedge x_2$, $x_2 \wedge x_3 = 0$. Finally, $x_1 \wedge x_3 \wedge x_4 = x_1 \wedge x_2 \wedge x_4$ and 
$$
x_1 \wedge x_2 \wedge x_3 = x_2 \wedge x_3 \wedge x_4 = x_1 \wedge x_2 \wedge x_3 \wedge x_4 = 0.
$$
Hence it has a basis $\{ 1, x_1, x_2, x_3, x_4, x_1 \wedge x_2, x_1 \wedge x_4, x_1 \wedge x_2 \wedge x_4 \}$ and dimension $8$. Thus, proposition  \ref{conj:mainresult}, together with \cite[Theorem 1.1]{BellamyNamikawa}, would imply that there are two $\Q$-factorial terminalizations of $\C^4 / D_m$. 

\subsection{The group $G_4$}

There is an error in the computation of the cohomology of the hyperplane arrangement associated to the group $G_4$ in example 4.3 of \cite{BellamyNamikawa}. The conclusion is correct though. We repeat the computation here. For the group $G_4$, the space $\mf{c} = \{ (\kappa_0,\kappa_1, \kappa_2 ) \ | \ \sum_i \kappa_i = 0 \}$ is the reflection representation for $W = \s_3$. It is known that $(\h \times \h^*) / G_4$ admits a symplectic resolution \cite{Singular}. The CM-hyperplanes were computed in the proof of \cite[Theorem 1.4]{MoPositiveChar}. The $T$-hyperplanes are the root hyperplanes
$$
\kappa_1 = 0, \quad \kappa_2 = 0, \quad \kappa_1 + \kappa_2 = 0. 
$$
The $F$-hyperplanes are 
$$
\kappa_1 - 2 \kappa_2, \quad \kappa_1 - \kappa_2, \quad 2 \kappa_1 - \kappa_2.
$$
Then the Orlik-Solomon algebra $H^{\idot}(\mf{c} \smallsetminus \mc{D},\C)$ is a graded quotient of the exterior algebra 
$$
\wedge^{\idot}(x_1,x_2,x_3,x_4,x_5,x_6).
$$
Using the computer algebra system MAGMA \cite{MAGMA}, one can compute that the Poincar\'e polynomial of $H^{\idot}(\mf{c} \smallsetminus \mc{D},\C)$ is $5t^2 + 6t + 1$. Hence there are 
$$
\frac{1}{|W|} \dim H^{\idot}(\mf{c} \smallsetminus \mc{D},\C)  = \frac{12}{6} = 2
$$
non-isomorphic symplectic resolutions of $(\h \times \h^*) / G_4$. This implies that the two symplectic resolutions constructed in \cite{LehnSorger} exhaust all symplectic resolutions.

\subsection{Exceptional complex reflection groups}

The Calogero-Moser families and the hyperplane arrangement $\mc{E}$ for many (20 out of 34) exceptional complex reflection groups have been explicitly computed by the third author \cite{Thiel:2015jl}, and by Bonnafé and the third author \cite{BT-CM-computational}. We refer to \textit{loc. cit.} for details about the computations. The explicit hyperplane arrangements are available as \textsc{Sage} files on the third author's website.  
\begin{table}[htbp] 
\begin{displaymath}
{\footnotesize
\hspace{-55pt}\begin{array}{r|ccccc} 
\textrm{Group} & |\mc{E}| & \textrm{Weyl group} & \textrm{Poincar\'e polynomial} & E & \textrm{Free?} \\
\hline
G_4   & 6 & \s_3  & (5t+1)(t+1) & 2  & \textrm{Yes} \\
G_5   & 33 & \s_3 \times \s_3  & (116t^2 + 21t + 1)(11t + 1)(t+1) & 92  & \textrm{No} \\
G_6   & 16 & \s_2 \times \s_3  & (8t + 1)(7t + 1)(t+1)& 12  & \textrm{Yes} \\
G_7   &61 & \s_2 \times \s_3 \times \s_3  & (98644t^4 + 18462t^3 + 1489t^2 + 60t + 1)(t+1) & 3296  & \textrm{No} \\
G_8   & 25 & \s_4  & (13t + 1)(11t + 1)(t+1) & 14  & \textrm{Yes} \\
G_9   & 54 & \s_2 \times \s_4  &  (6499t^3 + 983t^2 + 53t^1 + 1)(t+1) & 2  & \textrm{No} \\
G_{10}  & 111 & \s_3 \times \s_4   & (1001586t^4 + 107662t^3 + 4913t^2 + 110t + 1)(t+1) & 15476  & \textrm{No} \\
G_{11}  &  196 & \s_2 \times \s_3 \times \s_4  & (383999826t^5 + 25688824t^4 + 857259t^3 + 17047t^2 + 195t + 1)(t+1) & 2851133  & \textrm{No}\footnotemark \\
G_{13}  & 6 & \s_2 \times \s_2  & (5t+1)(t+1) & 3  & \textrm{Yes} \\
G_{14}  &  22& \s_2 \times \s_3 & (116t^2 + 21t + 1)(t+1) & 23  & \textrm{No} \\
G_{15}  &  65 & \s_2 \times \s_3 \times \s_2 & (13982t^3 + 1529t^2 + 32t + 1)(1+t) & 2596  & \textrm{No} \\
G_{20}  & 12 & \s_3  & (11t+1)(t+1) & 4  & \textrm{Yes} \\
G_{25}  &  12 & \s_3  & (11t+1)(t+1) & 4  & \textrm{Yes} \\
G_{26}  &  37 & \s_2 \times \s_3  & (335t^2 + 36t + 1)(t+1) & 62  & \textrm{No} \\
F_4 = G_{28}   &  8 & \s_2 \times \s_2  & (7t+1)(t+1) & 4  & \textrm{Yes} \\
\end{array} 
}
\end{displaymath} 
\caption{Data for exceptional complex reflection groups.} \label{exc_table}
\end{table}
\footnotetext{The computation of the Poincaré polynomial in case of $G_{11}$ took three weeks. There is, however, a simpler way to see that the hyperplane arrangement is not free, namely it has a non-free localization. We would like to thank M. Cuntz for pointing this out to us.}
Let $E := \frac{1}{|W|} \dim H^{\idot}( \mf{c} \smallsetminus \mc{E}; \C)$. Recall that it is an integer, counting the number of $\Q$-factorial terminalizations admitted by $V / \Gamma$. In Table \ref{exc_table} we list for each exceptional group which is not $2$-generated the Namikawa Weyl group, the Poincaré polynomial of $\mc{E}$, the integer $E$, and we list whether the hyperplane arrangement $\mc{E}$ is free or not. We recall that a hyperplane arrangement is \textit{free} if its module of derivations is free over the coordinate algebra of the ambient vector space, see \cite[\S4.2]{OrlikTeraoBook}. For the free arrangements the Poincaré polynomial factorizes into integral linear factors of the form $b_i t + 1$, and the $b_i$ are the \textit{exponents} of the arrangement, see \cite[4.137]{OrlikTeraoBook}.

{\small
\bibliography{biblo}{}
\bibliographystyle{abbrv}
}

\end{document}